\documentclass[11pt,twoside]{amsart}

\usepackage{amsmath}
\usepackage{amsthm}
\usepackage{amsfonts, amssymb}
\usepackage{mathrsfs}
\usepackage[all,line]{xy}
\usepackage{url}
\usepackage{graphics}
\usepackage{epstopdf}
\usepackage{comment}
\usepackage{xcolor}

\definecolor{thmcol}{RGB}{65, 102, 245}
\definecolor{citecol}{RGB}{65, 102, 245}
\definecolor{linkcol}{RGB}{65, 102, 245}
\definecolor{urlcol}{RGB}{65, 102, 245}

\usepackage[colorlinks=true]{hyperref}
\hypersetup{
    colorlinks,
    citecolor=citecol,
    linkcolor=linkcol,
    urlcolor=urlcol
}

\usepackage{latexsym}
\usepackage{graphicx}

\usepackage{tikz}
\usepackage{tkz-graph}
\usetikzlibrary{decorations.text,arrows,quotes}

\theoremstyle{plain}
\newtheorem{lemma}{Lemma}[section]
\newtheorem{proposition}[lemma]{Proposition}
\newtheorem{thm}[lemma]{Theorem}
\newtheorem*{thm*}{Theorem}
\newtheorem*{proposition*}{Proposition}
\newtheorem{corollary}[lemma]{Corollary}
\theoremstyle{remark}

\newtheorem{remark}[lemma]{Remark}

\theoremstyle{definition}
\newtheorem{definition}[lemma]{Definition}
\newtheorem{ej}[lemma]{Example}

\newtheorem*{thm1}{Theorem \ref{ch}}
\newtheorem*{thm2}{Theorem \ref{teocov}}
\newtheorem*{thm3}{Theorem \ref{asphericalcw}}
\newtheorem*{thm4}{Theorem \ref{acvstcg}}
\newtheorem*{thm5}{Theorem \ref{isobundles}}
\newtheorem*{prop1}{Proposition \ref{smooth}}

\def\Z{\mathbb{Z}}
\def\R{\mathbb{R}}

\def\tx{\widetilde{X}}
\def\ty{\widetilde{Y}}
\def\ta{\widetilde{A}}
\def\xtoi{X^{\mathrm I}}
\def\xytoi{(X\times Y)^{\mathrm I}}
\def\mtoi{M^{\mathrm I}}
\def\ytoi{Y^{\mathrm I}}
\def\atoi{A^{\mathrm I}}
\def\txtoi{\tx^{\mathrm I}}

\def\tatoi{\ta^{\mathrm I}}
\def\pix{\Pi(X)}

\begin{document}

\title[Topological Cartan-Hadamard Theorem and AC]{A topological version of the Cartan-Hadamard Theorem and asphericity complexity}
\author[E.G. Minian and J.M. Perez Garber]{Elías Gabriel Minian and Juan Martín Perez Garber}

\address{Departamento  de Matem\'atica-IMAS (CONICET)\\
 FCEyN, Universidad de Buenos Aires\\ Buenos
Aires, Argentina.}

\email{gminian@dm.uba.ar}
\email{jperezgarber@dm.uba.ar}
\thanks{The first author is researcher of CONICET and is partially supported by grant UBACYT 20020250100097BA}

\begin{abstract}
We investigate a topological version of the Cartan-Hadamard theorem that allows one to study asphericity of spaces via distinguished families of paths. A topological space has the distinguished path property (dpp) if there exists a continuous map from the space $\Pi(X)$ of homotopy classes of paths (relative endpoints) to the path space $X^I$ that is a right inverse to the canonical quotient map. For complete metric spaces endowed with a locally convex metric, the distinguished paths are precisely the local geodesics. We show that if $X$ has the dpp, then its universal cover is contractible and, in particular, $X$ is aspherical.  The space $\Pi(X)$ is a fiber bundle over $X$ whose fiber is the universal cover of $X$, and in the case of Riemannian manifolds of non-positive curvature it is naturally isomorphic to the tangent bundle. We prove that every aspherical  CW-complex that is locally finite or countable has the dpp, and this allows us to reinterpret asphericity of CW-complexes in terms of the existence of continuous sections. We also define and study the notion of asphericity complexity of spaces by means of local sections from $\Pi(X)$ to $\xtoi$ and relate it to the concept of equivariant topological complexity introduced by Colman and Grant.
\end{abstract}

\subjclass[2020]{55R10, 55M30, 53C23, 57M60, 57M10.} 

\keywords{Asphericity, topological complexity, non-positive curvature.}

\maketitle

\section{Introduction}

The Cartan-Hadamard theorem for locally convex metric spaces asserts that, if $X$ is a complete connected and locally convex metric space, then the universal cover $\widetilde X$ with the induced length metric is (globally) convex \cite{BH, BBI}. If we require extra hypotheses on $X$, we obtain stronger results. For instance, if $X$ is a complete Riemannian manifold of non-positive curvature, then $\widetilde X$ is diffeomorphic to $\R^n$. Similarly, if $X$ is locally CAT(0), then $\widetilde X$ is (globally) CAT(0) (see \cite{BH,gromov}). From the topological point of view, in any of these cases, $\widetilde X$ is contractible and, in particular, $X$ is aspherical. In fact, from the local convexity of the metric one can deduce the following two facts which, in turn, imply the contractibility of the universal cover: there exists a unique local geodesic in any homotopy class of paths in $X$ (relative endpoints), and locally, the geodesics depend continuously on their endpoints. This motivates the following construction. For a topological space $X$ that admits a universal cover, we consider the set $\pix$ of homotopy classes of paths in $X$ (relative endpoints) endowed with a convenient topology (see Section \ref{pix} below). The space $\pix$ is not new, it is the space of arrows of the topological fundamental groupoid of $X$ and appears, for example, in the work of Brown and Danesh-Naruie \cite{BD}. The map $s:\pix\to X$ defined by $s([\gamma])=\gamma(0)$ is a locally trivial fiber bundle over $X$ whose fiber is the universal cover $\widetilde X$. We use $\pix$ to state a topological version of the Cartan-Hadamard theorem. We say that $X$ has the distiguished path property (dpp) if 
there exists a continuous map $g:\pix\to\xtoi$ that is a section of the quotient map $q:\xtoi\to\pix$. The map $g$ chooses a {\it distinguised path} in every homotopy class. These paths play the rol of the local geodesics. The topological version of the Cartan-Hadamard theorem is proved in Section \ref{pix}.
\begin{thm1}
	If $X$ has the dpp, then $\tx$ is contractible and, in particular, $X$ is aspherical.
	\end{thm1}
The assignement $X\to \pix$ is functorial and has very nice properties. In Section \ref{basics} we prove that it preserves covering maps and deck transformation groups.
\begin{thm2}
If $p:Y\to X$ is a covering map, then $p_*:\Pi(Y)\to \pix$ is also a covering map. Moreover, the deck transformation groups $Deck(Y/X)$ and $Deck(\Pi(Y)/\pix)$ are isomorphic. Also, if $\tx$ is the universal cover of $X$, then $\Pi(\tx)$ is the universal cover of $\pix$.
\end{thm2}
We also show that it preserves CW structures and smooth structures. More specifically, for smooth manifolds we prove the following.
\begin{prop1}
	If $M$ is a smooth manifold of dimension $n$, then $\Pi(M)$ is a smooth manifold of dimension $2n$. Moreover $s:\Pi(M)\to M$ is a smooth bundle.
\end{prop1}
For complete Riemannian manifolds of non-positive curvature, we prove that it is isomorphic (as a fiber bundle) to the tangent bundle.

\begin{thm5}
	If $M$ is a complete Riemannian manifold of nonpositive curvature, then $TM$ and $\Pi(M)$ are isomorphic as fiber bundles over $M$.
\end{thm5}

In Section \ref{dppforcovers} we show that every aspherical CW-complex that is locally finite or countable has the dpp. Recall that a CW-complex $X$ is locally finite if every point is contained in a finite number of cells. It is said to be countable if it has countably many cells.
\begin{thm3}
	Let $X$ be a locally finite or countable aspherical CW-complex, then $X$ has the dpp.
\end{thm3}
This allows us to describe asphericity of CW-complexes in terms of the existence of a global continuous section.

In Section \ref{convexity} we study a topological notion of convexity for subspaces of spaces having the dpp, and in Section \ref{ac} we introduce and investigate the notion of asphericity complexity. The (normalized) asphericity complexity $AC(X)$ measures how far $X$ is from being aspherical. For example, a CW-complex $X$ is aspherical if and only if $AC(X)=0$. This notion is closely related to the notion of equivariant topological complexity introduced by Colman and Grant \cite{CG}. More concretely, we prove the following.
\begin{thm4}
	Let $X$ be a (good) space. Then $AC(X)=TC_G(\widetilde X)$, where $p:\tx\to X$ is the universal cover and $G=Deck(\tx/X)$.
	\end{thm4}
The paper concludes with the computation of the asphericity complexity of spheres and projective spaces and a formula for the asphericity complexity of the product of spaces.

\section{The distinguished path property and the covering bundle $\Pi(X)$}\label{pix}

We will work with topological spaces that admit a universal cover. Recall that a path connected space $X$ admits a universal covering if and only if it is locally path connected and semilocally simply connected. These conditions on $X$ imply that it has a basis for the topology consisting of path connected open subsets $U\subset X$ such that the inclusion $i:U\to X$ induces the trivial map $i_*=0:\pi_1(U,x)\to\pi_1(X,x)$ for every $x\in U$. For simplicity, such open subsets will be called {\it nice}, and the spaces which admit universal cover will be called {\it good}. Note that CW-complexes and manifolds are good.

From now on all spaces that we deal with are assumed to be good (in particular they are all path connected). We denote by $\xtoi$ the (free) path space of $X$ equipped with the compact-open topology.

Let $\pix=\{[\gamma],\ \gamma:I\to X\}$ be the set of homotopy classes (relative endpoints) of continuous paths in $X$. We endow $\pix$ with the following topology, which is similar to that of the universal covering of $X$, viewed as the space of homotopy classes of paths starting at a given point $x_0\in X$ (cf. \cite{mun}). Concretely, for every pair of nice open subsets $U,V\subseteq X$ and every continuous path $\gamma:I\to X$ with $\gamma(0)\in U$ and $\gamma(1)\in V$, define  $S(U,V,[\gamma])=\{[\gamma'],\ \gamma'=\delta_1*\gamma*\delta_2\}$ where $\delta_1$ and $\delta_2$ are paths contained in $U$ and $V$ respectively, and $\delta_1*\gamma*\delta_2$ denotes the concatenation of the paths (see Figure \ref{fig:suv}).

\begin{figure}[h]
\centering
	\includegraphics[scale=.2]{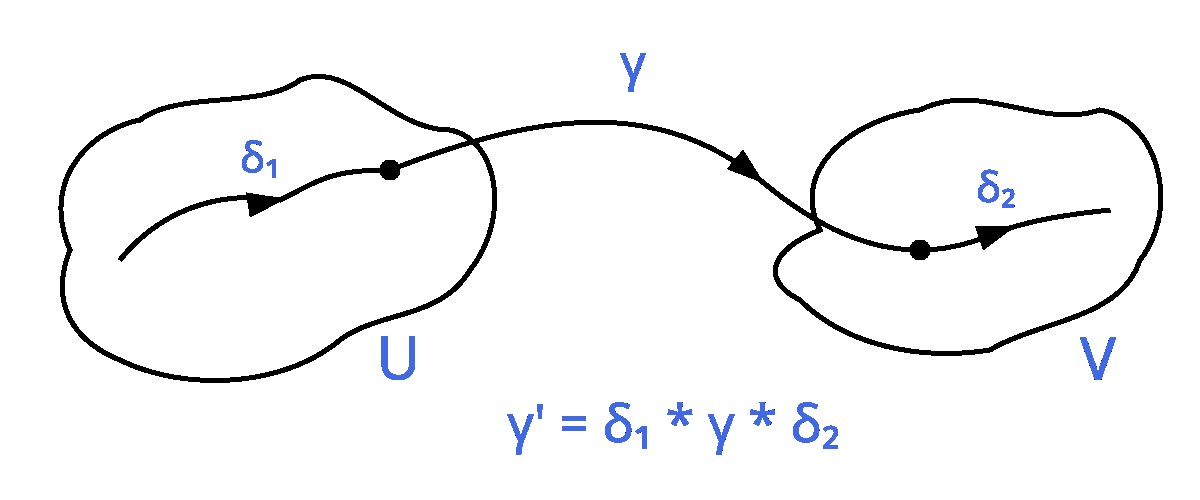}
\caption{}
\label{fig:suv}
\end{figure}

One can check that the sets $S(U,V,[\gamma])$ form a basis for a topology on $\Pi(X)$ and we equip $\pix$ with the topology generated by this basis. 

Note that the maps $$s,t:\pix\to X$$
defined by $s([\gamma])=\gamma(0)$ and $t([\gamma])=\gamma(1)$ are continuous.

\begin{remark}
	By construction, for every $x_0\in X$, the fibre $s^{-1}(x_0)=\widetilde{X}_{x_0}$ is the universal cover of $X$ (with base point $x_0$) and the restricion of the map $t$ to $s^{-1}(x_0)$ is the covering map $\widetilde{X}_{x_0}\to X$.
	\end{remark}
	
The space $\pix$ constructed above is just the space of arrows of the topological fundamental groupoid of $X$, studied by R. Brown and G. Danesh-Naruie in \cite{BD} (see also \cite{HH}). The following result was proved by Brown and Danesh-Naruie using a result of Ehresmann \cite{ehr}. We exhibit here a direct proof, since it will be used later in the paper.

\begin{proposition}
Let $X$ be a (good) space. Then	$s:\pix\to X$ is a fiber bundle over $X$ with fiber the universal cover $\tx$.
	\end{proposition}
	
\begin{proof}
	We show that for any nice open subset $U\subseteq X$ there exists a local trivialization map $\varphi:s^{-1}(U)\to U\times \tx$ such that the following diagram commutes.
	
		\begin{center}
		
		\begin{tikzpicture}[scale=0.5,
			node distance=4cm,
			every node/.style={font=\large}
			]

			\node (A) at (0,4)   {$s^{-1}(U)$};
			\node (B) at (6,4)   {$U \times \tx$};
			\node (C) at (3,0)   {$U$};

			\draw[->] (A) -- (B) node[midway, above] {$\varphi$};

			\draw[->] (A) -- (C) node[midway, left] {$s$};

			\draw[->] (B) -- (C) node[midway, right] {$pr_1$};

			\node at (3, 2.6) {$\equiv$};
			
		\end{tikzpicture} 
	\end{center}
	
Here $pr_1$ denotes the projection on the first coordinate. Given a nice open set $U\subset X$, fix a point $p\in U$. We identify $\tx=s^{-1}(p)$ and define $\varphi:s^{-1}(U)\to U\times \tx$ by $\varphi([\gamma])=(\gamma(0),[\lambda_{px}*\gamma])$, where $\lambda_{px}$ is any path in $U$ from $p$ to $x=\gamma(0)$.  Since $U$ is nice, $\varphi$ is well defined, and it is easy to see that it is a homeomorphism.
\end{proof}

The space $\pix$ will be called the {\it covering bundle} over $X$. By definition of the topology on $\pix$  one can deduce the following (see also \cite{HH}).

\begin{proposition}
	The map $q:\xtoi\to\pix$ defined by $q(\gamma)=[\gamma]$ is a quotient map.
	\end{proposition}

\begin{ej} The covering bundle of the $1$-dimensional sphere $S^1$ is trivial, i.e. $\Pi(S^1)=S^1\times\R$. One can easily check that the map $\varphi: \Pi(S^1)\to S^1\times\R$ defined by $$\varphi([\gamma])=(\gamma(0),(\widetilde{\gamma(0)^{-1}\gamma})(1)),$$ where $\widetilde{\gamma(0)^{-1}\gamma}$ is the lifting of the path $\gamma(0)^{-1}\gamma$ starting at $0\in\R$, is a homeomorphism, and it induces an isomorphism of bundles over $S^1$.
	\end{ej}
	
	The previous example can be generalized to any topological group.
	
	\begin{remark}
		If $G$ is a topological group, $\Pi(G)=G\times \widetilde G$ via $$\varphi([\gamma])=(\gamma(0), [\gamma(0)^{-1}\gamma]).$$ Here $\widetilde G$ is the universal cover of $G$ viewed as the space of homotopy classes of paths starting at the identity $1\in G$. Note that $\Pi(G)=G\times \widetilde G$ is also a topological group (we give $\widetilde G$ a group structure via$[\gamma][\beta]=[\gamma.\beta]$).
		\end{remark}

	We will show below that, for any complete Riemannian manifold of non-positive curvature $M$, the covering bundle $\Pi(M)$ is isomorphic (as a fiber bundle over $M$) to the tangent bundle $TM$. This will produce an infinite number of examples of spaces with non-trivial $\Pi(M)$.

\begin{definition}
	A space $X$ has the distinguished path property (dpp) if there exists a continuous map $g:\Pi(X)\to \xtoi$ which is right inverse to the quotient map $q$, i.e. if $[g([\gamma])]=[\gamma]$. In that case, the map $g:\Pi(X)\to\xtoi$ is a called a choice map.
	\end{definition}

\begin{ej}
	Connected complete metric spaces with locally convex metrics have the dpp. This follows from the two facts mentioned in the introduction: there exists a unique local geodesic in any homotopy class of paths in $X$, and locally, the geodesics depend continuously on their endpoints (see \cite[II.4.7]{BH} and \cite{BBI}).
	\end{ej}
	
The proof of the topological version of the Cartan-Hadamard theorem	follows the same ideas as the metric version (cf. \cite[II.4.5]{BH}): if $X$ has the dpp, one can use a choice map $g:\pix\to\xtoi$ to contract every homotopy class of paths to the class of a constant map.

\begin{thm}\label{ch}
	If $X$ has the dpp, then $\tx$ is contractible and, in particular, $X$ is aspherical.
	\end{thm}
	
		\begin{proof}
		Take a base point $x_0$ and let $\widetilde X=\widetilde{X}_{x_0}=s^{-1}(x_0)$ be the universal cover with base point $x_0$.	Let $r:\xtoi \times I \to \xtoi$ be the map $r(\gamma, s)=r_s(\gamma)$ with $r_s(\gamma)(t)=\gamma(t.s)$.

The map $H:\widetilde X\times I \to \widetilde X$ defined as
		$$H([\gamma],s)=[r_s(g[\gamma])]$$
		is a homotopy between the identity and the class of the constant map $x_0$. Note that $H$ is continuous since it is the co-restriction of the following composition of maps
		
	\begin{center}
	\begin{tikzpicture}[>=stealth, baseline=(current bounding box.center)]
		
		\node (A) at (0,0) {$\widetilde{X}\times I$};
		\node (B) at (3.6,0) {$\xtoi\times I$};
		\node (C) at (6.9,0) {$\xtoi$};
		\node (D) at (9.3,0) {$\Pi(X).$};
		
		\draw[->] (A) -- node[above] {$g\times \mathrm{id}$} (B);
		\draw[->] (B) -- node[above] {$r$} (C);
		\draw[->] (C) -- node[above] {$q$} (D);
	\end{tikzpicture}
\end{center}
	\end{proof}	
	
	The previous result asserts that if $X$ has the dpp then it is aspherical. In Section \ref{dppforcovers} we will prove that any aspherical CW-complex that is locally finite or countable has the dpp. We will show now that any contractible and locally path connected space (not necessarily with the same homotopy type as a CW-complex) has the dpp.
	
\begin{remark}
If $X$ is simply connected, then the map $s\times t:\pix\to X\times X$, defined by $$s\times t([\gamma])=(\gamma(0),\gamma(1))$$ is a homeomorphism.
\end{remark}

	\begin{proposition}
		If $X$ is contractible and locally path connected then it has the dpp.
		\end{proposition}
		
		\begin{proof}
			Let $H:X\times I\to X$ be a homotopy with $H_0=id_X$ and $H_1=c_0$ (the constant map $x_0$, for some $x_0\in X$).
			
		Take $\hat H:X\to \xtoi$, $\hat H(x)(t)=H(x,t)$ and define $g':X\times X\to \xtoi$ by $$g'(x,y)=\hat H(x)*\overline{\hat H}(y),$$ the concatenation of the path $\hat H(x)$ with the reverse of the path $\hat H(y)$.
		
		By the previous remark, $\pix$ is homeomorphic to $X\times X$. We obtain the desired map $g:\pix\to\xtoi$ by composing this homeomorphism with the map $g'$.
		\end{proof}
		
\section{Basic properties of $\pix$}\label{basics}

\begin{remark}
	The assignement $X\to\pix$ is functorial. A map $f:X\to Y$ induces a map $f_*:\pix\to\Pi(Y)$, $f_*([\gamma])=[f\gamma]$.
	\end{remark}
	
	Since $\pix$ locally looks like $X\times \tx$, it follows that, if $X$ is good, then $\Pi(X)$ is good. The next result asserts that $\Pi$  preserves covering maps and deck transformation groups.
	
	\begin{thm}\label{teocov}
		If $p:Y\to X$ is a covering map, then $p_*:\Pi(Y)\to \pix$ is also a covering map. Moreover, the deck transformation groups $Deck(Y/X)$ and $Deck(\Pi(Y)/\pix)$ are isomorphic. Also, if $\tx$ is the universal cover of $X$, then $\Pi(\tx)$ is the universal cover of $\pix$.
		\end{thm}
	\begin{proof}
		In order to prove that $p_*$ is a covering map, we have to find an evenly covered open set around any $[\gamma]\in\pix$. Given $[\gamma]$, take $U,V\subseteq X$ nice open subsets evenly covered by $p$ and such that $\gamma(0)\in U$ and $\gamma(1)\in V$. We will show that $S(U,V,[\gamma])$ is evenly covered by $p_*$.
		
		Since $U$ is evenly covered by $p$, then $\displaystyle p^{-1}(U)=\coprod_{\alpha\in\Lambda} U_\alpha$
		where
		$p_\alpha^U=p|_{U_\alpha}\colon U_\alpha\to U$
		is a homeomorphism for each $\alpha$. Let $\delta_\alpha$ be the lifting of the
		path $\gamma$ starting in
				$x_{\alpha}$, the unique point in $p^{-1}(\gamma(0))\cap U_\alpha$.
		Let $y_\alpha=\delta_\alpha(1)$.
		Since $V$ is evenly covered,
		$p^{-1}(V)=\coprod_{\alpha\in\Lambda} V_\alpha$
		with $y_\alpha\in V_\alpha$. Let
		$p_\alpha^V=p|_{V_\alpha}\colon V_\alpha\to V.$
		
		It is not difficult to see that
		\[
		(p_*)^{-1}\bigl(S(U,V,[\gamma])\bigr)
		=
		\coprod_{\alpha\in\Lambda}
		S(U_\alpha,V_\alpha,[\delta_\alpha]).
		\]
		
		Note that if $[\eta*\gamma*\mu]\in S(U,V,[\gamma])$, then
		$\bigl[(p_\alpha^U)^{-1}\eta
		*
		\delta_\alpha
		*
		(p_\alpha^V)^{-1}\mu
		\bigr]
		\in
		S(U_\alpha,V_\alpha,[\delta_\alpha])$
		is a lifting.
		
		It is clear that
		$p_*|_{S(U_\alpha,V_\alpha,[\delta_\alpha])}:S(U_\alpha,V_\alpha,[\delta_\alpha])\to S(U,V,[\gamma])$
		is a bijection and it is open by
		the definition of the topology on $\Pi(Y)$.
		This proves that
		$p_*:\Pi(Y)\to\Pi(X)$
		is a covering.
		
		Now let $G=Deck(Y/X)$. If
		$h:Y\to Y\in G$, then
		$h_*:\Pi (Y)\to\Pi (Y)$
		$\in Deck(\Pi(Y)/\Pi(X))$.
		Note that
		$p_*h_*=p_*$
		since $\Pi$ is functorial.
		
		This defines a group homomorphism
		\[
		L:Deck(Y/X)\longrightarrow Deck(\Pi(Y)/\Pi(X))
		\]
		\[
		h\longmapsto h_*.
		\]
		
		It is clear that $L$ is a monomorphism.
		We show that it is surjective. Given
		$\ell\in Deck(\Pi(Y)/\Pi(X))$, we show that there exists 	$\hat{\ell}\in Deck(Y/X)$ such that $L(\hat{\ell})=\hat{\ell}_*=\ell$.
		We can view $Y\subset\Pi(Y)$
		as the subspace of classes of constant paths $Y=\{[c_y],\ y\in Y\}$.
		The key point is that the restriction of $\ell$ to $Y$ induces an element
		$\hat{\ell}\in Deck(Y/X)$: if $[\beta]=\ell([c_y])$, then $[p\beta]=[c_{p(y)}]$ and since $p:Y\to X$ is a covering, $[\beta]=[c_{y'}]$ for some $y'\in Y$.
		It follows that $L(\hat{\ell})=\hat{\ell}_*=\ell$
		since $\hat{\ell}_*$ and $\ell$ coincide in the classes of constant paths and $\Pi(Y)$ is path connected.
		
		Finally, if $p:\tx\to X$ is the universal cover, then $p_*:\Pi(\tx)\to\pix$ is the universal cover as well, since, in this case, $\Pi(\tx)=\tx\times \tx$ is simply connected.
		\end{proof}
		
	\begin{corollary}\label{mismopiuno}
		Let $X$ be a space,then 
	\[
	\begin{aligned}
		\pi_1(\pix) &= \pi_1(X) \ \text{ and }\\
		\pi_n(\pix) &= \pi_n(X)\times\pi_n(X) \quad  \forall n\geq 2.
	\end{aligned}
	\]
		 In particular, $X$ is aspherical if and only if $\pix$ is aspherical.
		\end{corollary}
		
	\begin{remark}\label{longexact}
		 Note that Corollary \ref{mismopiuno} can also be deduced from the long exact sequence of homotopy groups
		\[
		\scalebox{0.95}{$
		\cdots \longrightarrow
		\pi_3(X) \longrightarrow
		\pi_2(\widetilde{X})
		\longrightarrow
		\pi_2(\Pi(X))
		\xrightarrow{\,s_*\,}
		\pi_2(X)
		\longrightarrow
		\pi_1(\widetilde{X})
		\longrightarrow
		\pi_1(\Pi(X))
		\xrightarrow{\,s_*\,}
		\pi_1(X)
		\longrightarrow
		0
		$}
		\]
		associated to the fibration $s:\pix\to X$ with fiber $\tx$. In particular, if $X$ is aspherical, $s:\pix\to X$ is a weak equivalence.
		\end{remark}
	
	We will show now that the functor $\Pi$ preserves CW-structures and smooth structures. 
	
	\begin{proposition}
		If $X$ is a locally finite or countable CW-complex, then $\pix$ is  a (locally finite or countable) CW-complex.
		\end{proposition}
		
		\begin{proof}
			If $X$ is locally finite, its universal cover $\tx$ is also locally finite, which implies that $\tx\times \tx$ (with the product topology) has a natural CW-complex structure. If $X$ is countable, then $\pi_1(X)$ is a countable group and $\tx$ is a countable CW-complex as well, and therefore $\tx\times \tx$, with the product topology, is a CW-complex. In both cases, $\tx$ is in fact a $G$-CW-complex, where $G=\pi_1(X)$ is viewed as the deck transformation group $Deck(\tx/X)$. And the free left action of $G$ on $\tx$ is a nice action, in the sense that the cells in $\tx$ are attached equivariantly, since its structure is lifted from the one of $X$ (see also \cite[Section II.1]{td}). Then $\Pi(\tx)=\tx\times \tx$ inherits also a nice $G$-action. In  this case, $G$ acts diagonally. Therefore $\pix=\Pi(\tx)/G$ has a well defined CW-structure.
			\end{proof}
			
			\begin{corollary}
				If $X$ is a locally finite or countable aspherical CW-complex, then $s:\pix\to X$ is a homotopy equivalence.
				\end{corollary}
				\begin{proof}
					This follows from the previous result, Remark \ref{longexact} and Whitehead's Theorem.
					\end{proof}
					
\begin{proposition}\label{smooth}
	If $M$ is a smooth manifold of dimension $n$, then $\Pi(M)$ is a smooth manifold of dimension $2n$. Moreover $s:\Pi(M)\to M$ is a smooth bundle.
\end{proposition}

\begin{proof}
We only have to prove that the transition maps between the local trivializations are smooth. Recall that, if $p\in M$ and $U$ is a nice open set around $p$, the local trivialization map is defined as $\varphi([\gamma])=(\gamma(0),[\lambda_{px}*\gamma])$ where $\lambda_{px}$ is any path in $U$ from $p$ to $x$. 

\begin{center}
	
	\begin{tikzpicture}[scale=0.5,
		node distance=4cm,
		every node/.style={font=\large}
		]

		\node (A) at (0,4)   {$S^{-1}(U)$};
		\node (B) at (6,4)   {$U \times \widetilde{M_p}$};
		\node (C) at (3,0)   {$U$};

		\draw[->] (A) -- (B) node[midway, above] {$\varphi$};

		\draw[->] (A) -- (C) node[midway, left] {$s$};

		\draw[->] (B) -- (C) node[midway, right] {$pr_1$};

		\node at (3, 2.6) {$\equiv$};
		
	\end{tikzpicture} 
\end{center}

The transition maps
$$\psi\varphi^{-1}:(U\cap V)\times \widetilde{M_p}\to (U\cap V)\times \widetilde{M_q}$$
are defined by $\psi\varphi^{-1}(x,[\gamma])=(x,[\lambda_{qx}*\lambda_{xp}*\gamma])$.

These maps are smooth because the first coordinate is the identity, while the composition of the second coordinate $(U\cap V)\times\widetilde{M_p}\to\widetilde{M_q}$ with the covering map $t:\widetilde{M_q}\to M$ is the map $(x,[\gamma])\to \gamma(1)$, which is precisely the covering map $\widetilde{M_p}\to M$.
\end{proof}

\begin{thm}\label{isobundles}
	If $M$ is a complete Riemannian manifold of nonpositive curvature, then $TM$ and $\Pi(M)$ are isomorphic as fiber bundles over $M$.
\end{thm}

\begin{proof}
	 For any $p\in M$, let $$\widehat{M_p}=\{c:I\to M,\ c \text{ is a local geodesic with } c(0)=p\}.$$ 

We show first that the map $\varphi:TM\to \displaystyle\bigcup_p \widehat{M_p}\subset \mtoi$ defined by $\varphi(p,v)=c_v$ (where $c_v$ is the unique geodesic with $c_v(0)=p$ and $c'_v(0)=v$) is a homeomorphism. To prove this, note first that the topology on $\mtoi$ coincides in this case with the one induced by the uniform (supremum) metric. In order to see that $\varphi$ is continuous one has to see that if $v_n\to v$ in $TM$, then $c_{v_n}\to c_v$ uniformly in $\mtoi$. This follows from the continuous dependence of solutions of differential equations on their initial conditions, together with the fact that, since $M$ is complete, $TM$ admits a natural Riemannian metric (the Sasaki metric) for which it is complete, and the exponential map $TM\to M$ is uniformly continuous on bounded subsets. In order to see that $\varphi^{-1}$ is continuous, take $(c_n)_{n \in \mathbb{N}}$ a sequence of geodesics converging uniformly to a geodesic $c$. We have to show that $\varphi^{-1}(c_n) = \dot{c}_n(0) \rightarrow \dot{c}(0) = \varphi^{-1}(c)$. Let us show that every subsequence of $(c_n)$ has a sub-subsequence that, under $\varphi^{-1}$, converges to $\dot{c}(0)$.
Let $(c_{n_k})_{k \in \mathbb{N}}$ be a subsequence. The sequence $(\dot{c}_{n_k}(0))_{k \in \mathbb{N}} \subseteq TM$ is bounded, since the set $\{c_{n_k}(0) \mid k \in \mathbb{N}\} \subseteq M \subseteq TM$ is bounded and the restriction of the Sasaki metric to each $T_p M$ is the original metric. Furthermore, since $\{d(\dot{c}_{n_k}(0), c_{n_k}(0)) \mid k \in \mathbb{N}\}$ is bounded, then $\{\dot{c}_{n_k}(0)\}$ is bounded, and therefore it has a convergent subsequence  $\dot{c}_{n_{k_j}}(0) \longrightarrow w \in TM$ (we use that $TM$ is complete and locally compact). Since $\varphi$ is continuous, $c_{n_{k_j}} = \varphi(\dot{c}_{n_{k_j}}(0)) \longrightarrow c_w$ uniformly, and therefore $w = \dot{c}(0)$ by uniqueness of the limit.

	 On the other hand, the Cartan-Hadamard theorem implies that the map $q$ restricted to 	$\displaystyle\bigcup_p \widehat{M_p}$ is a homeomorphism (see \cite[II.4.5]{BH}). Now take the composition.
	\[
\begin{tikzpicture}[>=stealth]
	
	\node (A) at (0,0) {$TM$};
	\node (B) at (3.5,0) {$\displaystyle\bigcup_{p}\widehat{M}_p$};
	\node (C) at (7,0) {$\Pi(M)$};
	\node (D) at (3.5,-2.2) {$M$};
	
	\draw[->] (A) -- node[above] {$\varphi$} (B);
	\draw[->] (B) -- node[above] {$q=g^{-1}$} (C);
	
	\draw[->] (A) -- (D);
	\draw[->] (C) -- (D);
	
	\node at (3.5,-1.0) {$\equiv$};
	
\end{tikzpicture}
\]
	\end{proof}	

\section{The dpp for coverings and aspherical CW-complexes}\label{dppforcovers}
	
	If $p:Y\to X$ is a covering map and $G=Deck(Y/X)$, then $\Pi(Y)$ has a natural $G$-action induced by composition: if $\varphi\in G$ and $[\beta]\in\Pi(Y)$, $\varphi[\beta]=[\varphi\beta]$.  Similary, $\ytoi$ is also a $G$-space. 
	 
	\begin{proposition}\label{dppcov}
		Suppose $X$ has the dpp. If $p:Y\to X$  is a covering, then $Y$ has the dpp. Moreover, there exists a choice map $g':\Pi(Y)\to \ytoi$ that is $G$-equivariant, where $G=Deck(Y/X)$.
	\end{proposition}

	\begin{proof}
		Let $g:\pix\to \xtoi$ be a choice map for $X$. Consider the fiber product $$\Pi(X)\underset{p}{\times}Y=\{([\gamma],y) \in \pix\times Y|\ p(y)=\gamma(0)\}.$$
		Note that there is a continuous map  $\ell:\xtoi\underset{p}{\times}Y\to \ytoi$  defined as
		$\ell(\alpha,y)=\widetilde{\alpha}$, where $\widetilde{\alpha}$ is the unique path in $Y$ such that $p\widetilde{\alpha}=\alpha$ and
		$\widetilde{\alpha}(0)=y$. Now consider the diagram
		\[
	\begin{tikzpicture}[>=latex,scale=0.9,transform shape]

		\node (A) at (0,0) {$\Pi(Y)$};
		
		\node (B) at (4.2,0)
		{$\Pi(X)\underset{p}{\times}Y$};
		
		\node (C) at (8.8,0)
		{$\xtoi\underset{p}{\times}Y$};
		
		\node (D) at (12.2,0)
		{$\ytoi$};
		
		\draw[->] (A) -- node[above] {$h$} (B);
		
		\draw[->] (B) -- node[above] {$g\times 1$} (C);
		
		\draw[->] (C) -- node[above] {$\ell$} (D);
		
		\draw[dashed,->]
		(A) to[bend right=18] node[below] {$g'$} (D);
		
	\end{tikzpicture}
	\]
		where $h([\gamma])=([p\gamma],\,\gamma(0))$. Let $g'$ be the composition $\ell\circ (g\times 1)\circ h$. Note that $g'$ is continuous and it is equivariant since, for any path $\beta$ in $X$ and any $\varphi\in Deck(Y/X)$, if $\widetilde{\beta}$ is the lifting of $\beta$ to a path in $Y$ starting at a point $y\in Y$, then $\varphi\widetilde{\beta}$ is the lifting of $\beta$ starting at $\varphi(y)$.
		\end{proof}

For regular coverings, the converse of the last proposition holds.

\begin{proposition}\label{regular}
	Let $p:Y\to X$ be a regular covering and let $G=Deck(Y/X)=Deck(\Pi(Y)/\Pi(X))$. If $Y$ admits a $G$-equivariant choice map $g':\Pi(Y)\to\ytoi$, then there exists a choice map $g:\Pi(X)\to\xtoi$ such that the following diagram commutes.
\[
\begin{tikzpicture}[>=stealth,baseline=(current bounding box.center)]
	
	\node (A) at (0,2) {$\Pi(Y)$};
	\node (B) at (3.2,2) {$\ytoi$};
	\node (C) at (0,0) {$\Pi(X)$};
	\node (D) at (3.2,0) {$\xtoi$};
	
	\draw[->] (A) -- node[above] {$g'$} (B);
	\draw[->] (A) -- node[left] {$p_*$} (C);
	\draw[->] (B) -- node[right] {$\tilde p$} (D);
	\draw[->] (C) -- node[below] {$g$} (D);
	
	\node at (1.6,1) {$\equiv$};
	
\end{tikzpicture},
\]
	where $\tilde p(\gamma)=p\gamma$.
	\end{proposition}

\begin{proof}
	
	Given $[\gamma]\in\Pi(X)$, define $g([\gamma])=\tilde p(g'[\widetilde{\gamma}])$ for any lifting $\widetilde{\gamma}$ of $\gamma$ (starting at any point in the fiber of $\gamma(0)$). We show that this is well defined. Suppose $[\gamma]=[\gamma']$ and let $\tilde\gamma$ and $\tilde\gamma'$ be liftings of $\gamma$ and $\gamma'$ with $\tilde\gamma(0)=y$ and $\tilde\gamma'(0)=y'$. Since $g'$ is equivariant, we only have to check that $[\tilde\gamma']=[\varphi\tilde\gamma]$ for some $\varphi\in Deck(Y/X)$. Now, since $p:Y\to X$ is a regular covering, there exists  $\varphi\in Deck(Y/X)$ with $\varphi(y)=y'$. This implies that  $[\tilde\gamma']=[\varphi\tilde\gamma]$.
	
 The continuity of $g$ follows from the fact that $p_*:\Pi(Y)\to\pix$ is a covering map.
\end{proof}

\begin{thm}\label{asphericalcw}
	Let $X$ be a locally finite or countable aspherical CW-complex, then $X$ has the dpp.
\end{thm}

\begin{proof}
Let $\tx$ be the universal cover of $X$ and $G=Deck(\tx/X)$. By Proposition \ref{regular}, it suffices to prove that $\tx$ admits a  $G$-equivariant choice map $g':\Pi(\tx)\to\txtoi$.

Note that $\tx$ inherits a CW-structure from $X$. Moreover it is a  $G$-complex in the sense of \cite[Chapter II]{td}. 
Since $X$ is locally finite or countable, then $\tx\times \tx$ is also a $G$-complex with the diagonal action. Note that $\Pi(\tx)$ is $G$-homeomorphic to $\tx\times \tx$ via $\varphi([\gamma])=(\gamma(0),\gamma(1))$.

Then it suffices to find a $G$-equivariant map $\tilde g:\tx\times \tx \to\txtoi$ that is right inverse to the map $q:\txtoi\to\tx\times\tx$, $q(\gamma)=(\gamma(0),\gamma(1))$. We use an equivariant version of a classical lifting result for CW-complexes (see also \cite{td}). Since $q$ is a fibration and a homotopy equivalence, there is a lifting $\tilde g$. And since $q$ is a $G$-map and $\tx\times\tx$ is a $G$-complex, we can get an equivariant lifting $\tilde g$.

\[
\begin{tikzpicture}[scale=0.50,
	node distance=4cm,
	every node/.style={font=\large}
	]
	
	\node (A) at (0,0)   {$\widetilde{X} \times \widetilde{X}$};
	\node (B) at (6,0)   {$\widetilde{X} \times \widetilde{X}$};
	\node (C) at (6,4)   {$\txtoi$};
	
	\draw[->] (A) -- (B) node[midway, below] {$\mathrm{id}$};
	
	\draw[->] (C) -- (B) node[midway, right] {$q$};
	
	\draw[->, dashed] (A) -- (C) node[midway, above left] {$\exists\,\tilde{g}$};
	
	\node at (3.8, 1.5) {$\equiv$};
	
\end{tikzpicture}
\]
\end{proof}	

\begin{corollary}\label{cwdpp}
	Let $X$ be a locally finite or countable CW-complex, then it is aspherical if and only if it has the dpp.
	\end{corollary}
	\begin{proof}
This follows from Theorem \ref{ch} and Theorem \ref{asphericalcw}.
\end{proof}

 Using Corollary \ref{cwdpp} we can reinterpret asphercity of (locally finite or countable) CW-complexes in terms of the existence of a certain section (after replacing $X$ by the CW-complex $\pix$).

\begin{remark}
If $f,g:X\to Y$ are homotopic, then so are $f_*,g_*:\pix\to\Pi(Y)$. A homotopy $H:X\times I\to Y$ from $f$ to $g$ induces a homotopy $\hat H:\pix\times I\to \Pi(Y)$ from $f_*$ to $g_*$, defined by $\hat H([\gamma],t)=[\varphi_t]$, with $\varphi_t(s)=H(\gamma(s),t)$. As a consequence, a homotopy equivalence $f:X\to Y$ induces a homotopy equivalence $f_*:\pix\to\Pi(Y)$. 
\end{remark}

It is  natural to ask whether the dpp is invariant under homotopy equivalences. By the previous remark, if $X\simeq Y$, then $\pix\simeq \Pi(Y)$. If we compose a choice map (section) $g_X:\pix\to \xtoi$ with a homotopy equivalence, we obtain a homotopy section $g'_Y:\Pi(Y)\to\ytoi$. We show next that the map $\ytoi\to\Pi(Y)$ is a fibration and, as a consequence, we can replace the homotopy section $g'_Y$ by a choice map (an actual section) $g_Y:\Pi(Y)\to\ytoi$ using the homotopy lifting property.

\begin{lemma}\label{gfib}
Let $G$ be a group and let $E$ and $B$ be $G$-spaces such that the action of $G$ on both spaces is properly discontinuous. If $p:E\to B$ is a $G$-fibration, then the induced map $\overline p:E/G\to B/G$ is a fibration.  
\end{lemma}
\begin{proof}
This follows from the proof of \cite[Theorem 5]{gevo}. The argument applies verbatim. Note that the compactness of $G$ in the statement of  \cite[Theorem 5]{gevo} is not used, and the good definition of the map $\lambda^*$ in the proof follows, in our case, for the unique path lifting of the covering $B\to B/G$.
\end{proof}
\begin{proposition}
The map $q:\xtoi\to\pix$ is a fibration.
\end{proposition}

\begin{proof}
Consider the universal cover $\tx$. Note that the map $p:\txtoi\to\tx\times\tx$, $p(\gamma)=(\gamma(0),\gamma(1))$ is a $G$-fibration where $G=Deck(\tx/X)$ (see, for example, \cite{CG,grant,td}). On the other hand, $\xtoi=(\tx/G)^{\rm I}=\txtoi/G$ since $\tx\to X=\tx/G$ is a covering. Moreover, via the identification $\Pi(\tx)=\tx\times\tx$,  $\pix=\Pi(\tx)/G=(\tx\times\tx)/G$. Now the result follows from Lemma \ref{gfib}.
\end{proof}
\begin{corollary}
Let $X$ and $Y$ be homotopy equivalent spaces, then $X$ has the dpp if and only if $Y$ has the dpp.
\end{corollary}

In section \ref{ac} we will generalize this result and relate if to the notion of equivariant topological complexity introduced by Colman and Grant (see \cite{CG,grant}).

\section{Aspherical subcomplexes and convexity}\label{convexity}

It is well-known that if $X$ is an aspherical CW-complex of dimension $2$ and $A\subseteq X$ is a connected subcomplex such that the inclusion $i:A\to X$ is $\pi_1$-injective (i.e. $i_*:\pi_1(A)\to\pi_1(X)$ is injective), then $A$ is also aspherical. This follows from the fact that, for any connected $2$-complex $Y$, $Y$ is aspherical if and only if the second homology group $H_2(\widetilde Y)$ of its universal cover is trivial. For higher dimensions this is not longer true. For example, the inclusion $S^2\subset D^3$ of the $2$-sphere in the $3$-disc induces an isomorphism in $\pi_1$ but $S^2$ is not aspherical. This is because $S^2$ is not a convex subcomplex of the disc.

\begin{remark}
Let $X$ be a locally finite or countable CW-complex and $A\subseteq X$ a subcomplex. If the inclusion $i:A\to X$ is $\pi_1$-injective then $\Pi(A)\subseteq \Pi(X)$ is a subcomplex. 
\end{remark}

\begin{definition}
Let $X$ be a locally finite or countable CW-complex and $A\subseteq X$ a subcomplex such that the inclusion $i:A\to X$ is $\pi_1$-injective. If $X$ is aspherical, by Theorem \ref{asphericalcw} there is a choice map $g:\Pi(X)\to \xtoi$. We say that $A$ is a convex subcomplex (with respect to $g$) if the image of the restriction of $g$ to $\Pi(A)$ lies in $\atoi$. 
\end{definition}

\begin{remark}
Note that if $A$ is a convex subcomplex for some choice map  $g:\Pi(X)\to \xtoi$, then $A$ is aspherical. 
\end{remark}

\begin{thm}
	Let $X$ be a locally finite or countable aspherical CW-complex and let $A\subset X$ be a subcomplex such that $i_*:\pi_1(A)\to\pi_1(X)$ is an isomorphism. Then $A$ is aspherical if and only if it is convex for some choice map $g:\Pi(X)\to \xtoi$.
\end{thm}

\begin{proof}
Suppose $A$ is aspherical. Let $G=Deck(\widetilde X/X)$. Pick a choice map $g_A:\Pi(A)\to\atoi$ and extend $g_A$ to a $G$-equivariant choice map $\tilde g_{A}:\Pi(\widetilde A)\to\tatoi$. Consider the following commutative diagram
\[
\begin{tikzpicture}[>=stealth, baseline=(current bounding box.center)]
	
	\node (A) at (0,2) {$\Pi(\widetilde{A})$};
	\node (B) at (3.5,2) {$\txtoi$};
	\node (C) at (0,0) {$\Pi(\widetilde{X})=\widetilde{X}\times\widetilde{X}$};
	\node (D) at (3.5,0) {$\widetilde{X}\times\widetilde{X}$};
	
	\draw[->] (A) -- node[above] {$\tilde i\circ \tilde{g}_{A}$} (B);
	\draw[->] (A) -- (C);
	\draw[->] (B) -- (D);
	\draw[->] (C) -- node[above] {$id$} (D);
	
	\node at (1.75,1) {$\equiv$};
	
\end{tikzpicture}
\]
Note that $\Pi(\widetilde A)\to\Pi(\tx)$ is a $G$-cofibration because $i_*:\pi_1(A)\to\pi_1(X)$ is an isomorphism. Then, since $\txtoi \to\tx\times\tx$ is $G$-fibration and a $G$-homotopy equivalence, there exists an equivariant map $\tilde g:\Pi(\tx)\to\txtoi$ such that the following diagram commutes.
\[
\begin{tikzpicture}[x=3cm, y=2.5cm, >=stealth]
	\node (TL) at (0, 1) {$\Pi(\widetilde{A})$};
	\node (TR) at (1, 1) {$\txtoi$};
	\node (BL) at (0, 0) {$\Pi(\tilde{X})$};
	\node (BR) at (1, 0) {$\tilde{X} \times \tilde{X}$};
	
	\draw[->] (TL) -- (TR) node[midway, above] {$\tilde{\imath} \circ \tilde{g}_A$};
	\draw[->] (TL) -- (BL);
	\draw[->] (BL) -- (BR) node[midway, above] {$id$};
	\draw[->] (TR) -- (BR);
	
	\draw[->, dashed] (BL) -- (TR) node[midway, above left] {$\tilde{g}$};

\end{tikzpicture}
\]
The equivariant map $\tilde g$ induces a choice map $g:\Pi(X)\to\xtoi$ that extends $g_A:\Pi(A)\to\atoi$.
\end{proof}

\section{Asphericity complexity}\label{ac}

In this section we introduce and investigate a new invariant, the asphericity complexity $AC(X)$ of a space $X$. This invariant measures how far the space $X$ is from being aspherical.
	
\begin{definition}
	Let $X$ be a (good) topological space. The (normalized) asphericity complexity of $X$ is the minimum interger $k$ such that $\pix$ can be covered by open subsets $U_0,\ldots,U_k$, each of them admitting a choice map (continuous section) $s_j:U_j\to \xtoi$, with $qs_j=i_j:U_j\hookrightarrow \pix$. Here $q:\xtoi\to\pix$ is the quotient map and $i_j$ denotes the inclusion. We denote $AC(X)=k$.
\end{definition}

\begin{remark}
	Note that $AC(X)=0$ if and only if $X$ has the dpp. In particular, for locally finite or countable CW-complexes $X$, $AC(X)=0$ if and only if $X$ is aspherical.
\end{remark}

\begin{remark}\label{acforsc}
	If $X$ is simply connected, $q:\xtoi\to \pix$ is, up to homeomorphism, the canonical fibration $p:\xtoi\to X\times X$. Therefore, in this case,  $AC(X)=TC(X)$, the (normalized) topological complexity of $X$ (see \cite{farber}).
	\end{remark}

The notion of asphericity complexity is related to that of equivariant topological complexity introduced by Colman and Grant \cite{CG} (see also \cite{grant}). Recall that the equivariant topological complexity $TC_G(X)$
of a $G$-space $X$ is the minimum interger $k$ such that $X\times X$ can be covered by invariant open subsets $U_0,\ldots,U_k$, each of them admitting a $G$-map  $s_j:U_j\to \xtoi$, with $ps_j=i_j:U_j\hookrightarrow X\times X$. Here $p:\xtoi\to X\times X$ is the cannonical $G$-fibration.

\begin{thm}\label{acvstcg}
	Let $X$ be a (good) space. Then $AC(X)=TC_G(\widetilde X)$, where $p:\tx\to X$ is the universal cover and $G=Deck(\tx/X)$.
	\end{thm}
	
	\begin{proof}
		 Consider the map $p_*:\Pi(\tx)=\tx\times \tx\to \pix$. Suppose first that $U_0,\ldots,U_k$ are open subsets of $\pix$ and  $s_j:U_j\to \xtoi$ are local sections. We proceed similarly as in the proof of Proposition \ref{dppcov}. For each $j$, take the $G$-invariant open subset $\widetilde{U_j}=p_*^{-1}(U_j)\subseteq \tx\times\tx$, and define a map $\tilde s_j:\widetilde{U_j}\to\txtoi$ as the composition
		 	\[
		 \begin{tikzpicture}[>=latex,scale=0.9,transform shape]

		 	\node (A) at (0,0) {$\widetilde{U_j}\vphantom{\underset{p}{\times}}$};
		 	
		 	\node (B) at (4.2,0)
		 	{$U_j\underset{p}{\times}\tx$};
		 	
		 	\node (C) at (8.8,0)
		 	{$\xtoi\underset{p}{\times}\tx$};
		 	
		 	\node (D) at (12.2,0)
		 	 {$\txtoi\vphantom{\underset{p}{\times}}$};
		 	
		 	\draw[->] (A) -- node[above] {$h$} (B);
		 	
		 	\draw[->] (B) -- node[above] {$s_j\times 1$} (C);
		 	
		 	\draw[->] (C) -- node[above] {$\ell$} (D);
		 	
		 	\draw[dashed,->]
		 	(A) to[bend right=18] node[below] {$\tilde s_j$} (D);
		 	
		 \end{tikzpicture}
		 \]
		 The maps $h$ and $l$ are defined similarly as in the proof of Proposition \ref{dppcov}. 
		 
		 Suppose now that $\widetilde U_0,\ldots,\widetilde U_k$ are open subsets of $\tx\times\tx$ and  $\tilde s_j:\widetilde U_j\to \txtoi$ are local sections. We proceed as in the proof of Proposition \ref{regular}. For each $j$, take the open subset $U_j=p_*(\widetilde U_j)\subseteq \pix$, and define a map  $s_j:U_j\to \xtoi$ similarly as in Proposition \ref{regular}.
		 \end{proof}
		 
Note that, by definition, $AC(X)=secat(q)$ the sectional category of the fibration $q:\xtoi\to\pix$ (see \cite{CLOT}). A homotopy equivalence $f:X\to Y$ induces a commutative diagram
\[
\begin{tikzpicture}[baseline=(current bounding box.center)]
    \node (XI) at (0,1.5) {$X^I$};
    \node (YI) at (3,1.5) {$Y^I$};
    \node (PX) at (0,0) {$\Pi(X)$};
    \node (PY) at (3,0) {$\Pi(Y)$};

    \draw[->] (XI) -- node[above] {$\hat f$} (YI);
    \draw[->] (PX) -- node[below] {$f_*$} (PY);
    \draw[->] (XI) -- node[left] {$q_X$} (PX);
    \draw[->] (YI) -- node[right] {$q_Y$} (PY);
\end{tikzpicture}
,\]
where $\hat f$ and $f_*$ induced by the map $f:X\to Y$. Note that $\hat f$ and $f_*$ are homotopy equivalences as well. It follows that $secat(q_X)=secat(q_Y)$. Therefore we have proved the following.

\begin{proposition}
	If $X$ and $Y$ are homotopy equivalent spaces, then $AC(X)=AC(Y)$.
\end{proposition}

\begin{ej}
	(Asphericity complexity of spheres) Since $S^1$ is aspherical, then $AC(S^1)=0$. For $n\geq 2$, by Remark \ref{acforsc}, $AC(S^n)=TC(S^n)$, and by \cite[Theorem 8]{farber} we have
	\[
	AC(S^n) = \begin{cases} 
		0 & \text{for } n = 1, \\ 
		1 & \text{for } n \text{ odd}, \, n \neq 1, \\ 
		2 & \text{for } n \text{ even}. 
	\end{cases}
	\] 
\end{ej}
	\begin{ej}
		(Asphericity complexity of projective spaces) Let $\R P^n$ be the $n$-dimensional real projective space ($n\geq 2$). By Theorem \ref{acvstcg}, $AC(\R P^n)=TC_{\Z_2}(S^n)$, where $\Z_2$ acts antipodally on $S^n$. By \cite[Lemma 4.1]{GGTX} we have
		\[
		AC(\R P^n) = \begin{cases} 
			1 & \text{for } n \text{ odd},  \\ 
			2 & \text{for } n \text{ even}. 
		\end{cases}
		\] 
	\end{ej}

We prove next that, for CW-complexes $X$ of dimension $2$, $AC(X)$ can take only two values: if $X$ is not aspherical, $AC(X)=2$. Of course this is not longer true for higher dimensions (see Example \ref{prodspheres}). We will use the following well known result of Schwarz \cite{sch} (see also \cite{james}).

\begin{proposition}[Schwarz]\label{schwarz}
If $B$ is a CW-complex, the sectional category of a fibration $q:E\to B$ with fiber $F$ satisfies
$$secat(q)\leq \frac{\dim(B)+1}{conn(F)+2},$$
where $conn(F)$ denotes the connectivity of the fiber $F$.
\end{proposition}

\begin{proposition}\label{twocomp}
Let $X$ be a locally finite or countable $2$-complex. Then
	\[
		AC(X) = \begin{cases} 
			0 & \text{if } X \text{ is aspherical},  \\ 
			2 & \text{if } X \text{ is not aspherical}. 
		\end{cases}
		\] 
\end{proposition}
\begin{proof}
If $X$ is not aspherical, take the universal cover $\tx$. Since $\tx$ is a simply connected $2$-complex, then $\tx\simeq \bigvee_{j\in J}S^2$ with $J\neq\emptyset$ because $\tx$ is not contractible. Then 
$$AC(X)=TC_G(\tx)\geq TC(\tx)=TC(\bigvee_{j\in J}S^2)\geq TC(S^2)=2,$$
where $G=Deck(\tx/X)$.

On the other hand, $AC(X)=secat(q:\xtoi\to\pix)$. Note that the fiber $F$ of the fibration $q$ is the space of loops in $X$ (at some base point $x_0$) that are nullhomotopic. By the lifting property of the covering, $F=\Omega \tx$ (the loop space of $\tx$). Since $\tx$ is simply connected, $\pi_0(F)=\pi_1(\tx)=0$. Therefore $conn(F)\geq 0$ and, by Proposition \ref{schwarz}, $secat(q)\leq 5/2$ (since $\pix$ is a $4$-dimensional CW-complex). It follows that $AC(X)=2$.
\end{proof}

\begin{remark}
If $X$ is a locally finite or countable CW-complex of dimension $n$, then $AC(X)\leq n$. This follows from Proposition \ref{schwarz}, using the same argument as in the proof of Proposition \ref{twocomp}.
\end{remark}

We finish the paper studying the asphericity complexity of a product of spaces. 

\begin{proposition}\label{acproduct}
	If $X$ and $Y$ are locally finite or countable CW-complexes, then $AC(X\times Y)\leq AC(X)+AC(Y)$.
	\end{proposition}
\begin{proof}
	The proof is almost identical to the one given in \cite[Theorem 11]{farber} for the topological complexity of the product. Suppose $AC(X)=n$ and $AC(Y)=m$. Let $U_0,\ldots,U_n\subseteq \pix$ and $V_0,\ldots,V_m\subseteq \Pi(Y)$ be open covers with $s_i:U_i\to\xtoi$ and $s'_j:V_j\to \ytoi$ local sections.
	
	Since $\pix$ and $\Pi(Y)$ are CW-complexes, then they are paracompact spaces and we can choose partitions of the unity $f_i:\pix\to\R$ and $g_j:\Pi(Y)\to\R$ subordinate to the coverings $\{U_i\}$ and $\{V_j\}$. 
	
	For $S\subset \{0,\ldots,n\}$ and $T\subset\{0,\ldots,m\}$, we define the open subsets $$W(S,T)\subset \Pi(X\times Y)=\pix\times\Pi(Y)$$
	exactly as in the proof of  \cite[Theorem 11]{farber}. 
	
	Now, for $k=2,\ldots,n+m+2$, we take the union $W_k$ of all $W(S,T)$ with $|S|+|T|=k$. Then $W_2,\ldots,W_{n+m+2}$ is an open cover of $\Pi(X\times Y)$, and each $W_k$ has a section $$\sigma_k:W_k\to \xytoi=\xtoi\times\ytoi,$$ described in terms of $s_i\times s'_j:U_i\times V_j\to\xtoi\times \ytoi$.
	\end{proof}
	
If $X$ is a contractible space, then for any space $Y$, $TC(X\times Y)=TC(Y)$, since $X\times Y\simeq Y$. By Proposition	\ref{acproduct}, if $X$ is an aspherical CW-complex $AC(X\times Y)\leq AC(Y)$. We show that, in fact, they are equal.
\begin{proposition}
	If $X$ is a locally finite or countable aspherical CW-complex, then $AC(X\times Y)=AC(Y)$ for any locally finite or countable CW-complex $Y$.
	\end{proposition}
	\begin{proof}
		By Proposition \ref{acproduct}, we have  $AC(X\times Y)\leq AC(Y)$. We prove that  $AC(X\times Y)\geq AC(Y)$. Let $\tx$ and $\ty$ be the universal covers of $X$ and $Y$. Note that, since $X$ is aspherical, then $\tx$ is contractible. Let $H=Deck(\tx/X)$ and $K=Deck(\ty/Y)$, then $G=H\times K=Deck(\tx\times\ty/X\times Y)$. Since $\tx$ is contractible, the map $f:\tx\times\ty\to \{*\}\times \ty$ defined by $f(x,y)=(*,y)$ is a homotopy equivalence and is $K$-invariant. Here we view $K$ as the subgroup $1\times K\subset G$. Then $f$ is a $K$-homotopy equivalence and, therefore, $TC_K(\tx\times \ty)=TC_K(\ty)=AC(Y)$. By \cite[Corollary 5.4]{CG}, $TC_G(\tx\times \ty)\geq TC_K(\tx\times \ty)$. Then, $AC(X\times Y)=TC_G(\tx\times \ty)\geq TC_K(\tx\times \ty)=AC(Y)$.
		\end{proof}

Since $S^1$ is aspherical, $AC(S^1\times Y)=AC(Y)$ for any $Y$. Using this, and the computation of the topological complexity of any product of spheres proved in \cite{BGRT}, we deduce the following formula.

\begin{ej}\label{prodspheres}
	(Asphericity complexity of products of spheres) Let $X=S^{m_1}\times S^{m_2}\times\ldots\times S^{m_n}$. Since $AC(S^1\times Y)=AC(Y)$ for any $Y$, we remove from $X$ all occurrences of $S^1$, and consider  the product of the remaining spheres $X'=S^{r_1}\times\ldots\times S^{r_s}$ (with $s\leq n$). Since $r_j\geq 2$ for all $j$, $X'$ is simply connected and then, $AC(X')=TC(X')$. By \cite[Corollary 3.12]{BGRT} we have
	$$AC(S^{m_1}\times S^{m_2}\times\ldots\times S^{m_n})=s+t,$$
	where $s$ is the number of spheres of dimension greater than $1$ and $t$ is the number of spheres of even dimension.
			\end{ej}

\end{document}